\documentclass{amsart}
\usepackage{color} 
\usepackage{verbatim}
\usepackage{rotating} 
\usepackage{amssymb}
\usepackage{mathrsfs,amsfonts,amsmath,amssymb,epsfig,amscd,xy,amsthm,pb-diagram,bm,soul} 

\newtheorem{theorem}{Theorem}[section]
\newtheorem{proposition}[theorem]{Proposition}

\newtheorem{lemma}[theorem]{Lemma}

\newtheorem{corollary}[theorem]{Corollary}

\newtheorem{definition}[theorem]{Definition}

\theoremstyle{plain}

\theoremstyle{remark}

\newtheorem{remark}[theorem]{Remark}

\newcommand{\twopartdef}[4]
{
	\left\{
		\begin{array}{ll}
			#1 & \mbox{if } #2 \\
			#3 & \mbox{if } #4
		\end{array}
	\right.
}

\newcommand{\C}{{\mathbb C}}

\newcommand{\Q}{{\mathbb Q}}

\newcommand{\Z}{{\mathbb Z}}
\newcommand{\N}{{\mathbb N}}

\newcommand{\fr}{\mathfrak r}

\newcommand{\fs}{\mathfrak s}

\newcommand{\Qbar}{\bar{\Q}}

\DeclareMathOperator{\Gal}{Gal}

\DeclareMathOperator{\Norm}{N}

\newcommand{\bP}{{\mathbb P}}

\newcommand{\bfx}{{\mathbf x}}
\newcommand{\bfy}{{\mathbf y}}
\newcommand{\bfX}{{\mathbf X}}

\newcommand{\bfa}{{\mathbf a}}
\newcommand{\bfb}{{\mathbf b}}

\newcommand{\bmmu}{{\bm{\mu}}}

\newcommand{\cM}{\mathcal{M}}

\newcommand{\cL}{\mathcal{L}}

\newcommand{\scrS}{\mathscr{S}}

\newcommand{\scrI}{\mathscr{I}}

\newcommand{\scrB}{\mathscr{B}}

\newcommand{\scrP}{\mathscr{P}}

\definecolor{myblue}{rgb}{0.6, 0.9, 1}
\definecolor{mygreen}{rgb}{0,0,1}
\definecolor{orange}{rgb}{0.8,0,0.2}

\usepackage[bookmarks,colorlinks,breaklinks]{hyperref}  
\hypersetup{linkcolor=blue,citecolor=blue,filecolor=blue,urlcolor=blue} 

\author{Avinash Kulkarni}
\address{
Avinash Kulkarni\\
Department of Mathematics\\
Simon Fraser University\\ 
Burnaby, BC V5A 1S6, Canada}
\email{akulkarn@sfu.ca}

\author{Niki Myrto Mavraki}
\address{
Niki Myrto Mavraki\\
Department of Mathematics\\
University of British Columbia\\ 
Vancouver, BC V6T 1Z2, Canada}
\email{myrtomav@math.ubc.ca}
\urladdr{www.math.ubc.ca/\~{}myrtomav}

\author{Khoa D.~Nguyen}
\address{
Khoa D.~Nguyen \\
Department of Mathematics\\
University of British Columbia\\
And Pacific Institute for The Mathematical Sciences\\ 
Vancouver, BC V6T 1Z2, Canada}
\email{dknguyen@math.ubc.ca}
\urladdr{www.math.ubc.ca/\~{}dknguyen}


\begin{document}
	\title[Algebraic Approximations to Linear Combinations of Powers]{Algebraic Approximations to Linear Combinations of Powers: an Extension of Results by Mahler and Corvaja-Zannier}
	
	\date{26 November, 2015}
	
	\begin{abstract}
	For every complex number $x$, let $\Vert x\Vert_{\Z}:=\min\{|x-m|:\ m\in\Z\}$. Let $K$ be a number field, let $k\in\N$, and let $\alpha_1,\ldots,\alpha_k$
	be non-zero algebraic numbers. In this paper,
	we completely solve
	the problem of the existence of $\theta\in (0,1)$ such that
	there are infinitely many tuples
	$(n,q_1,\ldots,q_k)$ 
	satisfying 
	$\Vert q_1\alpha_1^n+\ldots+q_k\alpha_k^n\Vert_{\Z}<\theta^n$
	where $n\in\N$ and 
	$q_1,\ldots,q_k\in K^*$ 
	 having 
	small logarithmic height
	compared to $n$. In the special case when $q_1,\ldots,q_k$ 
	have the form $q_i=qc_i$ for fixed $c_1,\ldots,c_k$,
	our work yields results on algebraic 
	approximations of 
	$c_1\alpha_1^n+\ldots+c_k\alpha_k^n$ of the form
	$\displaystyle \frac{m}{q}$ 
	with $m\in \Z$ and $q\in K^*$ (where $q$ has small logarithmic height compared to $n$).
	Various results on linear recurrence sequences also follow
	as an immediate consequence.
	The case $k=1$ and $q_1$ is essentially a rational integer
	was obtained by Corvaja and Zannier and settled a 
	long-standing question of Mahler. The use of the Subspace
	Theorem based on work of Corvaja-Zannier together with
	several modifications play an important role
	in the proof of our results.
	\end{abstract}

	\maketitle{}
	
  \section{Introduction}\label{sec:intro}
  Throughout this paper, let $\N$ denote
  the set of positive integers 
  and fix an embedding from $\Qbar$ to $\C$. For every
  complex number $x$, let $\Vert x\Vert_{\Z}$
  denote the distance to the nearest integer:
  $$\Vert x \Vert_{\Z}:=\min\{\vert x-m\vert:\ m\in \Z\}.$$ 
  
	In 1957, Mahler \cite{Mahler1957} used Roth's theorem (or more precisely, Ridout's non-archimedean version of Roth's theorem) to prove that 
	for $\alpha\in\Q\setminus \Z$ satisfying 
	$\alpha>1$ and for  
	$\theta\in (0,1)$, there are only
	finitely many $n\in\N$ such that 
	$\Vert\alpha^n\Vert_{\Z}<\theta^n$. Consequently, the
	number $g(k)$ in Waring's Problem satisfies:
	$$g(k)=2^k+\left\lfloor \left(\frac{3}{2}\right)^k\right\rfloor-2$$
    except for possibly finitely many $k$. Mahler also asked
    for which algebraic numbers $\alpha$ the above conclusion
    remains valid. In 2004, by ingenious applications of the Subspace Theorem, Corvaja and Zannier obtained the following
    answer to Mahler's question \cite[Theorem~1]{CZ-ActaMath2004}:
		
	\begin{theorem}[Corvaja-Zannier]\label{thm:CZ}
	Let $\alpha>1$ be a real algebraic number. Assume
	that for some $\theta\in (0,1)$, there are
	infinitely many $n\in\N$ such that  
	$\Vert \alpha^n\Vert_{\Z}<\theta^n$.
	Then there is $d\in\N$ such 
	that $\alpha^d$ is a Pisot number.
	\end{theorem}
	
	Note that a Pisot number is a real algebraic integer 
	greater than 1 while all the other conjugates
	have modulus less than 1. By considering the trace of
	$\alpha^{dn}$ for $n\in\N$, it is clear that the conclusion 
	of Theorem~\ref{thm:CZ} is optimal. Moreover, the condition
	that $\alpha$ is algebraic is necessary as illustrated
	in an example constructed by Corvaja and Zannier \cite[Appendix]{CZ-ActaMath2004}.

	By regarding $\{\alpha^n:\ n\in\N\}$ as
	a very special instance of linear recurrence
	sequences, it is natural to ask for an extension of Theorem~\ref{thm:CZ} to an arbitrary linear recurrence
	sequence
	of the form
	$\{Q_1(n)\alpha_1^n+\ldots+Q_k(n)\alpha_k^n:\ n\in\N\}$ 
	where $\alpha_i\in \Qbar^*$ and $Q_i(x)\in\Qbar[x]\setminus\{0\}$ for $1\leq i\leq k$. \emph{Without loss of generality,
	we can assume that $\vert\alpha_i\vert \geq 1$
	for $1\leq i\leq k$.}
	In fact, we will solve this problem for
	the more general sequence of the form $q_1\alpha_1^n+\ldots+q_k\alpha_k^n$ where $q_1,\ldots,q_k$ are non-zero 
	algebraic numbers having small logarithmic height compared to $n$. 
	
	Let $G_{\Q}$ be the absolute Galois group of $\Q$
	and let $\bmmu$ be the group of roots of unity. We have:
	\begin{definition}\label{def:non-degenerate}
	The tuple of non-zero algebraic numbers
	$(\alpha_1,\ldots,\alpha_k)$ is called non-degenerate
	if $\displaystyle\frac{\alpha_i}{\alpha_j}\notin \bmmu$
	for $1\leq i\neq j\leq k$.
	\end{definition}
	This definition does not rule out the possibility 
	that some $\alpha_i$ is a root of unity.
	When working with sums of the form $q_1\alpha_1^n+\ldots+q_k\alpha_k^n$, we can assume
	that $(\alpha_1,\ldots,\alpha_k)$ is non-degenerate without loss of generality. Indeed, suppose $\displaystyle \frac{\alpha_{k}}{\alpha_{k-1}}=\zeta$ is an $m$-th root of unity.
	For $0\leq r\leq m-1$, we restrict to $n\in \N$
	congruent to $r$ modulo $m$, and write
	$n=r+m\tilde{n}$. Then the sum
    $q_1\alpha_1^r(\alpha_1^{m})^{\tilde{n}}+q_2\alpha_2^r(\alpha_2^m)^{\tilde{n}}+\ldots+(q_{k-1}+\zeta^rq_k)\alpha_{k-1}^r(\alpha_{k-1}^m)^{\tilde{n}}$ is equal to and has less terms than the original 
    sum $q_1\alpha_1^n+\ldots+q_k\alpha_k^n$.

	The following definition generalizes the notion
	of Pisot numbers:
	\begin{definition}\label{def:Pisot tuple}
	Let $(\beta_1,\ldots,\beta_k)$ be a tuple of of distinct 
	non-zero 
	algebraic numbers. Write: 
	$$B:=\{\beta\in \Qbar^*\setminus\{\beta_1,\ldots,\beta_k\}:\ \beta=\sigma(\beta_i)\ \text{for some $\sigma\in G_{\Q}$
	and $1\leq i\leq k$}\}.$$
	Then the tuple $(\beta_1,\ldots,\beta_k)$ is called
	pseudo-Pisot if $\displaystyle\sum_{i=1}^k\beta_i+\sum_{\beta\in B}\beta\in \Z$ and
	$\vert\beta\vert <1$ for every $\beta\in B$. Moreover,
	if $\beta_i$ is an algebraic integer for $1\leq i\leq k$
	then the tuple $(\beta_1,\ldots,\beta_k)$ is called
	Pisot.
	\end{definition}
	Note that when $k=1$, our definition of being pseudo-Pisot  here is slightly different from Corvaja-Zannier \cite[p.~176]{CZ-ActaMath2004}
	since we do not require that $\vert\beta_i\vert > 1$
	for $1\leq i\leq k$. Such generality is necessary in
	order to deal with the
	condition $\vert \alpha_i\vert\geq 1$
	with the possibility of an equality.

	Let $h$ denote the 
	absolute logarithmic Weil height (see \cite{BG} or Section~\ref{sec:Subspace}). By a  
	\emph{sublinear function}, we mean a function
	$f:\N\rightarrow (0,\infty)$ satisfying $\displaystyle\lim_{n\to\infty} \frac{f(n)}{n}=0$. 
	Let $K$ be a number field, 
	our main result is the following:
	\begin{theorem}\label{thm:main1}
	Let
	$k\in\N$,
	let $(\alpha_1,\ldots,\alpha_k)$ be a non-degenerate
	tuple of algebraic numbers satisfying
	$\vert\alpha_i\vert \geq 1$
	for $1\leq i\leq k$, and let $f$ be a sublinear function.
	Assume that for some $\theta\in (0,1)$,
	the set $\cM$ of $(n,q_1,\ldots,q_k)\in \N\times (K^*)^{k}$ satisfying:
	$$\left\Vert \sum_{i=1}^k q_i\alpha_i^n\right\Vert_{\Z} <\theta^n\ \text{and}\ \max_{1\leq i\leq k} h(q_i)<f(n)$$
	is infinite. For all but finitely
	many $(n,q_1,\ldots,q_k)\in \cM$, the following
	hold:
	\begin{itemize} 
		\item [(i)] The tuple $(q_1\alpha_1^n,\ldots,q_k\alpha_k^n)$ is pseudo-Pisot.
		\item [(ii)] $\alpha_i$ is an algebraic integer for
		$1\leq i\leq k$.
		\item [(iii)] For $\sigma\in G_{\Q}$ and $1\leq i,j\leq k$, $\sigma(q_i\alpha_i^n)=q_j\alpha_j^n$
		if and only if $\displaystyle \frac{\sigma(\alpha_i)}{\alpha_j}\in\bmmu$.
		\item [(iv)] For $\sigma\in G_{\Q}$ and $1\leq i\leq k$, if $\displaystyle\frac{\sigma(\alpha_i)}{\alpha_j}\notin \bmmu$ for $1\leq j\leq k$ then 
		$\vert \sigma(\alpha_i)\vert<1$. 
	\end{itemize}
	\end{theorem}

	When $k=1$, properties (i) and (ii) recover the conclusion
	in Theorem~\ref{thm:CZ}.  As is the case with 
	the Corvaja-Zannier theorem, it is not difficult to
	see that properties (i)-(iv) above are sufficient. In fact
	$q_1\alpha_1^n+\ldots+q_k\alpha_k^n$ is 
	a subsum of an integer-valued trace and, 
	according to properties (iii) and (iv), any term that does
	not appear in this subsum must have the form 
	$\sigma(q_i\alpha_i^n)$ with
	$\vert \sigma(\alpha_i)\vert <1$. Now since
	$h(\sigma(q_i))<f(n)=o(n)$, we have that 
	$\vert \sigma(q_i\alpha_i^n)\vert < \theta^n$
	for some $\theta\in (0,1)$ when $n$ is sufficiently large.
	This explains why we have
	$\left\Vert\displaystyle \sum_{i=1}^k q_i\alpha_i^n\right\Vert_{\Z} <\theta^n$.
	
 	\begin{remark}\label{rem:mild condition} 
	As mentioned before, the condition 
	$\vert \alpha_i\vert \geq 1$ for $1\leq i\leq k$ is 
	non-essential. Indeed if, say,
	$\vert\alpha_i\vert <1$ then 
	$\vert q_i\alpha_i^n\vert<\theta^n$
	 for every $\theta\in (\vert\alpha_i\vert,1)$ when $n$ is 
	 sufficiently large. This follows from
	 the condition that $\vert q_i\vert =e^{o(n)}$. Then we can  
	 simply consider the smaller subsum without the term 
	 $q_i\alpha_i^n$. If $\alpha_i$ is a root of unity then (by choosing $j=i$), Property (iii) implies that
	 for all but finitely many $(n,q_1,\ldots,q_k)\in\cM$, $\sigma(q_i\alpha_i^n)=q_i\alpha_i^n$ for 
	 every $\sigma$ which means $q_i\alpha_i^n\in \Q$.
	\end{remark}

	For every polynomial $Q(x)\in \Qbar[x]$ and $\sigma\in G_\Q$,
	let $\sigma(Q)$ be the polynomial obtained by applying
	$\sigma$ to the coefficients of $Q$. For
	every rational function $R(x)\in \Qbar(x)$, write 
	$R(x)=\displaystyle\frac{Q_1(x)}{Q_2(x)}$, then we define
	$\sigma(R)=\displaystyle\frac{\sigma(Q_1)}{\sigma(Q_2)}$. This is 
	independent of the choices of $(Q_1,Q_2)$. 
	 We have the following immediate consequences:
	\begin{corollary}\label{cor:1}
	Let $k\in\N$ and let $(\alpha_1,\ldots,\alpha_k)$ be
	as in Theorem~\ref{thm:main1}. 
	Let $a_n:=R_1(n)\alpha_1^n+\ldots+R_k(n)\alpha_k^n$
	for $n\in\N$ where $R_1(x),\ldots,R_k(x)\in\Qbar(x)\setminus\{0\}$. Assume that for some $\theta\in (0,1)$, 
	the
	set 
	$\cM_{\theta}=\{n\in\N: \Vert a_n\Vert_{\Z}<\theta^n\}$
	is infinite. Then the following hold:
	\begin{itemize}
		\item [(i)] $\alpha_i$ is an algebraic integer for $1\leq i\leq k$.
		\item [(ii)] For $\sigma\in G_{\Q}$ and 
		$1\leq i\leq k$, if $\displaystyle \frac{\sigma(\alpha_i)}{\alpha_j}\notin \bmmu$ for 
		$1\leq j\leq k$ then $\vert\sigma(\alpha_i)\vert <1$. 
		\item [(iii)] For $\sigma\in G_{\Q}$ and $1\leq i,j\leq k$, if 
		$\displaystyle\frac{\sigma (\alpha_i)}{\alpha_j}\in\bmmu$ then the rational function $\displaystyle \frac{R_j}{\sigma(R_i)}$ is constant and 
	$\displaystyle\frac{R_j}{\sigma(R_i)}=\displaystyle\left(\frac{\sigma(\alpha_i)}{\alpha_j}\right)^n$
	for all but finitely many $n\in \cM_{\theta}$.
	\end{itemize}
	\end{corollary}
	
	For an arbitrary linear recurrence sequence, we
	obtain the following result in the style of the classical
	Skolem-Mahler-Lech theorem.
	\begin{corollary}\label{cor:2}
	Let $(a_n)_{n\in\N}$ be an arbitrary linear recurrence 
	sequence over $\Qbar$. Assume that for some 
	$\theta\in (0,1)$, the set $\cM_{\theta}$ as defined in 
	Corollary~\ref{cor:1} is infinite. Then there is $\theta_0\in (0,1)$ such that for every $\tilde{\theta}\in (\theta_0,1)$, the set $\cM_{\tilde{\theta}}$ is
	the union of a finite set and finitely many arithmetic
	progressions.
	\end{corollary}	   
	   
 The Subspace Theorem
 obtained by Schmidt and extended by Schlickewei and Evertse
 (see \cite{Schmidt72}, \cite{ESquantitative}, \cite[Chapter~7]{BG} and the references there)
 plays a crucial role in the proof of our results. 
 Our applications
 of the Subspace Theorem are based on the paper 
 \cite{CZ-ActaMath2004} together with several modifications. 
 We conclude this introduction with a brief 
 comparison to the compelling results of \cite{CZ-ActaMath2004}.
 The paper of Corvaja-Zannier settles the case $k=1$
 and
 $q_1$ is of the form $q\delta$ where $q\in\Z$
 and $\delta$ is a fixed algebraic number. In this paper, 
 we investigate an arbitrary sum ($k\geq 1$) and
 allow the  $q_i$'s to be in an arbitrary number field.
 While the key construction of
 linear forms in \cite[Lemma~3]{CZ-ActaMath2004} is
 rather neat, the construction in this paper is a bit more 
	complicated due to the extra generality noted above and so  requires
	some technical modifications (see Section~\ref{sec:technical}).

	{\bf Acknowledgments.} We wish to thank Professor Pietro Corvaja for useful comments. The third author is grateful to Professor Umberto 
	Zannier for useful conversations at
	the Mathematical Sciences Research Institute in March 2014.	
	The first author is partially supported by NSERC and the third author is partially supported by a UBC-PIMS fellowship.
	
	\section{Some classical applications of the Subspace Theorem}\label{sec:Subspace}	
  	Let $M_{\Q}=M_{\Q}^{\infty}\cup M_{\Q}^0$ where
  	$M_{\Q}^0$ is the set of $p$-adic valuations
  	and $M_{\Q}^{\infty}$
  	is the singleton consisting of the usual archimedean 
  	valuation. More generally, for every number field $K$,
  	write $M_K=M_K^\infty\cup M_K^0$ where
  	$M_K^\infty$ is the set of archimedean places and 
  	$M_K^0$ is the set of finite places. For every
  	$w\in M_K$, let $K_w$ denote the completion of 
  	$K$ with respect to $w$ and denote
  	$d(w)=[K_w:\Q_v]$ where $v$ is the restriction of
  	$w$ to $\Q$. Following \cite[Chapter~1]{BG},
  	for every $w\in M_K$ restricting to $v$ on $\Q$,
  	we normalize $\vert \cdot\vert_w$ as follows:
  	$$\vert x\vert_w = \vert \Norm_{K_w/\Q_v}(x) \vert_v^{1/[K:\Q]}.$$
  	Let $m\in\N$, for every vector $\bfx=(x_0,\ldots,x_m)\in K^{m+1}\setminus\{\mathbf 0\}$, let $\tilde{\bfx}$ denote
  	the corresponding point in $\bP^{m}(K)$. For every $w\in M_K$,
  	denote $\vert \bfx\vert_w:=\displaystyle\max_{0\leq i\leq m} \vert x_i\vert_w$.
  	For $P \in \bP^m(\Qbar)$, let $K$ be a number
  	field such that $P=\tilde{\bfx}$
  	for some $\bfx\in K^{m+1}\setminus\{\mathbf 0\}$
  	and define:
  	$$H(P)=\prod_{w\in M_K} \vert \bfx\vert_w.$$
  	Define $h(P)=\log (H(P))$. For $\alpha\in K$, 
  	write $H(\alpha):=H([\alpha:1])$
  	and write $h(\alpha):=\log(H(\alpha))$.

  	Now we present the celebrated Subspace Theorem.
  	The readers are referred to \cite[Chapter~7]{BG} and
  	the references there for more details.
 	\begin{theorem}\label{thm:Subspace}
 	Let $n\in\N$, let $K$ be a number field,
 	and let $S\subset M_K$ be finite. For every $v\in S$, let 
 	$L_{v0},\ldots,L_{vn}$ be linearly independent 
 	linear forms in the variables $X_0,\ldots,X_n$
 	with $K$-algebraic coefficients in $K_v$. For every 
 	$\epsilon>0$, the
 	solutions $\bfx\in K^{n+1}\setminus\{\mathbf{0}\}$
 	of the inequality:
 	$$\displaystyle\prod_{v\in S}\prod_{j=0}^n \displaystyle\frac{\vert L_{vj}(\bfx)\vert_v}{\vert \bfx\vert_v}\leq H(\tilde{\bfx})^{-n-1-\epsilon}$$
 	are contained in finitely many hyperplanes of
 	$K^{n+1}$.
 	\end{theorem}

    Variants of the following applications of the 
    Subspace Theorem are
  	well-known. We include the proofs due to the lack of an 
  	immediate reference.
  	\begin{proposition}\label{prop:general SML}
  	Let $(\alpha_1,\ldots,\alpha_k)$ be a non-degenerate 
  	tuple of non-zero algebraic numbers, let $f$ be a sublinear function, and let $K$ be a number field. Then there are 
  	only finitely many tuples $(n,b_1,\ldots,b_k)\in \N\times (K^*)^k$ satisfying:
  	$$b_1\alpha_1^n+\ldots+b_k\alpha_k^n=0\ \text{and}\ \max_{1\leq i\leq k}h(b_i)<f(n).$$
  	\end{proposition}
  	\begin{proof}
  	Assume there is a counter-example given by $(K,f,k,\alpha_1,\ldots,\alpha_k)$ where $k$ is minimal. This implies  that there are infinitely many such tuples $(n,b_1,\ldots,b_k)$. By the height inequality on the $b_i$'s and the Northcott  property, we must have that $n$ is unbounded in any infinite collection of such tuples.
  	By extending
  	$K$ if necessary, we can choose a finite subset $S$ of $M_K$ containing $M_K^{\infty}$ such that 
  	$\alpha_i$ is an $S$-unit for $1\leq i\leq k$. We will prove that 
  	$k\geq 3$; obviously $k>1$. If $k=2$, we have
  	$\displaystyle\frac{b_1}{b_2}=-\left(\frac{\alpha_2}{\alpha_1}\right)^n$. By non-degeneracy, 
  	$h(\alpha_2 / \alpha_1)>0$ while $h(b_1/b_2)<2f(n)=o(n)$
  	which gives a contradiction when $n$ is sufficiently large.

  	Write $\bfx=(x_1,\ldots,x_{k-1}):=(b_1\alpha_1^n,\ldots,b_{k-1}\alpha_{k-1}^n)$.
  	For $v\in S$, choose $i(v)$ with $1\leq i(v)\leq k-1$
  	such that $\vert \bfx\vert_v=\vert x_{i(v)}\vert_v$. There
  	are $(k-1)^{\vert S\vert}$ possibilities
  	for the choices of the $i(v)$'s, hence we may assume that
  	some \emph{fixed} choice is used for infinitely many
  	tuples $(n,b_1,\ldots,b_k)$, from now on we will focus on 
  	these tuples only. For $v\in S$ and $1\leq j\leq k-1$,
  	define:
  	$$L_{vj}(X_1,\ldots,X_{k-1})=\twopartdef{X_j}{j\neq i(v)}{X_1+\ldots+X_{k-1}}{j=i(v)}$$
  	
  	Since the $\alpha_i$'s are $S$-units, by the product
  	formula, we have:
  	\begin{equation}\label{eq:Subspace SML0}
  	\prod_{v\in S}\prod_{j=1}^{k-1}\displaystyle\frac{\vert L_{vj}(\bfx)\vert_v}{\vert\bfx\vert_v}
  	=\frac{\displaystyle\prod_{v\in S}\prod_{j=1}^{k}\vert b_j\alpha_j^n\vert_v}{\left(\displaystyle\prod_{v\in S} \vert \bfx\vert_v\right)^k}\leq \left(\displaystyle\prod_{v\in S} \vert \bfx\vert_v \right)^{-k}\prod_{j=1}^k H(b_j).  	
  	\end{equation} 
  	
  	Again, using the fact that the $\alpha_{i}$'s are
  	$S$-units, we have:
  	\begin{equation}\label{eq:Subspace SML1}
  		\displaystyle\prod_{v\in S} \vert \bfx\vert_v=\frac{H(\tilde{\bfx})}{\displaystyle\prod_{v\notin S}\max_{1\leq j\leq k-1}\vert b_j\vert_v}\geq\frac{H(\tilde{\bfx})}{\displaystyle\prod_{j=1}^{k-1} H(b_j)}. 
  	\end{equation}
  	
  	From \eqref{eq:Subspace SML0} and \eqref{eq:Subspace SML1},
  	we have:
  	\begin{equation}\label{eq:Subspace SML2}
  	\prod_{v\in S}\prod_{j=1}^{k-1}\displaystyle\frac{\vert L_{vj}(\bfx)\vert_v}{\vert\bfx\vert_v}\leq H(\tilde{\bfx})^{-k}\left(\prod_{j=1}^k H(b_j)\right)^{k+1}.
  	\end{equation}
  	
  	Write $C=H([\alpha_1:\ldots:\alpha_{k-1}])$ then $C>1$ thanks to non-degeneracy of $(\alpha_1,\ldots,\alpha_{k-1})$. We have:
  	\begin{equation}\label{eq:height1 SML}
	\begin{split}
	\frac{H([\alpha_1^n:\ldots:\alpha_{k-1}^n])}{H([b_1^{-1}:\ldots:b_{k-1}^{-1}])}&\leq H(\tilde{\bfx})=H([b_1\alpha_1^n:\ldots:b_{k-1}\alpha_{k-1}^n])\\
	&\leq H([\alpha_1^n:\ldots:\alpha_{k-1}^n])H([b_1:\ldots:b_{k-1}]) 
	\end{split}
	\end{equation}
	which implies
	\begin{equation}\label{eq:height2 SML} 
	C^n\left(\displaystyle\prod_{j=1}^k H(b_j)\right)^{-1}\leq H(\tilde{\bfx})\leq C^n\displaystyle\prod_{j=1}^k H(b_j).
	\end{equation}

  	Since $\displaystyle\prod_{j=1}^k H(b_j)\leq e^{kf(n)}=e^{o(n)}$, inequalities \eqref{eq:Subspace SML2}
  	and \eqref{eq:height2 SML} give that for every $\epsilon\in (0,1/2)$
  	we have:
  	$$\prod_{v\in S}\prod_{j=1}^{k-1}\frac{\vert L_{vj}(\bfx)\vert_v}{\vert\bfx\vert_v}\leq H(\tilde{\bfx})^{-k+\epsilon}\leq H(\tilde{\bfx})^{-k+1-\epsilon}$$ 
  	when $n$ is sufficiently large. Hence Theorem~\ref{thm:Subspace}  
  	gives that there exist a subset $\{i_1,\ldots,i_{\ell}\}$
  	of $\{1,\ldots,k-1\}$ and $c_{i_1},\ldots,c_{i_\ell}\in K^*$
  	such that:
  	$$c_{i_1}b_{i_1}\alpha_{i_1}^n+\ldots+c_{i_\ell}b_{i_\ell}\alpha_{i_\ell}^n=0$$
  	for infinitely many tuples $(n,b_{i_1},\ldots,b_{i_\ell})$.
  	By modifying the sublinear function $f$, 
  	the infinitely many tuples $(n,c_{i_1}b_{i_1},\ldots,c_{i_\ell}b_{i_\ell})$ yield that there is a counter-example
  	with parameter $\ell\leq k-1<k$, contradicting the
  	minimality of $k$. 
  	\end{proof}
  	
  	The next result is a slight modification
  	of \cite[Lemma~1]{CZ-ActaMath2004}:
  	\begin{proposition}\label{prop:CZLemma1}
  	Let $k\in\N$, let $K$ be
  	a number field, let $S$ be a finite subset of $M_K$ containing $M_K^{\infty}$, and let $\lambda_1,\ldots,\lambda_k$ be non-zero elements of $K$. Fix $w\in S$
  	and $\epsilon>0$. Let $\scrS$ be an infinite set of
  	solutions $(u_1,\ldots,u_k,b_1,\ldots,b_k)$
  	of the inequality:
  	$$\left\vert \displaystyle\sum_{j=1}^k\lambda_jb_ju_j\right\vert_w\leq \max\{\vert b_1u_1\vert_w,\ldots,\vert b_ku_k\vert_w\}(\displaystyle\prod_{j=1}^k H(b_j))^{-k-1-\epsilon}H([u_1:\ldots:u_k:1])^{-\epsilon}$$
  	where $u_j$ is an $S$-unit and $b_j\in K^*$ for $1\leq j\leq k$. Then there exists a non-trivial
  	linear relation of the form $c_1b_1u_1+\ldots+c_kb_ku_k=0$
  	where $c_i\in K$ for $1\leq i\leq k$ satisfied
  	by infinitely many elements of $\scrS$.
  	\end{proposition} 
  	\begin{proof}
  	When $k=1$, the given inequality together with the Northcott property imply that the set $\scrS$ must be finite. Hence we may assume $k\geq 2$. We now apply the Subspace Theorem as
  	in the proof of \cite[Lemma~1]{CZ-ActaMath2004}. Without loss of generality and by replacing $\scrS$ by an infinite subset, we may assume:
  	$$\vert b_1u_1\vert_w = \max\{\vert b_1u_1\vert_w,\ldots,\vert b_ku_k\vert_w\}$$
  	for every $(u_1,\ldots,u_k,b_1,\ldots,b_k)\in \scrS$. 
  	For
  	$v\in S$ and $1\leq j\leq k$,  if $(v,j)\neq (w,1)$ define $L_{vj}(X_1,\ldots,X_k)=X_j$; define $L_{w1}(X_1,\ldots,X_k)=\lambda_1X_1+\ldots+\lambda_kX_k$. Write
  	$\bfx=(b_1u_1,\ldots,b_ku_k)$, we have:
  	\begin{equation}\label{eq:new CZ Lemma1}
  	\begin{split}
  	\displaystyle\prod_{v\in S} \prod_{j=1}^k \frac{\vert L_{vj}(\bfx)\vert_v}{\vert\bfx\vert_v}&=\frac{\displaystyle\prod_{v\in S}\prod_{j=1}^k \vert b_ju_j\vert_v}{\vert b_1u_1\vert_w}\vert \sum_{j=1}^k\lambda_jb_ju_j\vert_w\left(\prod_{v\in S}\vert\bfx\vert_v\right)^{-k}\\
  	&\leq \frac{(\prod_{j=1}^k H(b_j))(\prod_{j=1}^k H(b_j))^{-k-1-\epsilon}H([u_1:\ldots:u_k:1])^{-\epsilon}}{\left(\displaystyle\prod_{v\in S}\vert\bfx\vert_v\right)^k}
  	\end{split}
  	\end{equation}
  	As in the proof of Proposition~\ref{prop:general SML},
  	we have:
  	$\displaystyle\prod_{v\in S}\vert \bfx\vert_v\geq
  	\frac{H(\tilde{\bfx})}{\prod_{j=1}^k H(b_j)}$. Together
  	with \eqref{eq:new CZ Lemma1}, we have:
  	\begin{equation}\label{eq:new CZ Lemma2}
  	\displaystyle\prod_{v\in S} \prod_{j=1}^k \frac{\vert L_{vj}(\bfx)\vert_v}{\vert\bfx\vert_v}\leq H(\tilde{\bfx})^{-k}
  	\left(\prod_{j=1}^k H(b_j)\right)^{-\epsilon} H([u_1:\ldots:u_k:1])^{-\epsilon}.
  	\end{equation}
  	Using 
  	$H(\tilde{\bfx})\leq H([u_1:\ldots:u_k])
  	\displaystyle\prod_{j=1}^k H(b_j)\leq H([u_1:\ldots:u_k:1])\displaystyle\prod_{j=1}^k H(b_j)$, we obtain:
  	$$\displaystyle\prod_{v\in S} \prod_{j=1}^k \frac{\vert L_{vj}(\bfx)\vert_v}{\vert\bfx\vert_v}\leq H(\tilde{\bfx})^{-k-\epsilon}.$$
  	Theorem~\ref{thm:Subspace} yields
  	the desired conclusion.
  	\end{proof}
  	
	\section{The Key Technical Results Toward the Proof of Theorem~\ref{thm:main1}}\label{sec:technical}
	The main technical results
	obtained in this section play a crucial role
	in the proof of Theorem~\ref{thm:main1}.
	We define an equivalence relation $\approx$ on $\Qbar^*$ as 
	follows: $\alpha\approx \beta$
	if there is $\sigma\in G_{\Q}$ such that
	$\displaystyle \frac{\alpha}{\sigma(\beta)}\in \bmmu$.
	Let $K$ be a number field and write $d=[K:\Q]$. We have the following:
		\begin{lemma}\label{lem:coefficients}
		Let $\{\gamma_1,\ldots,\gamma_d\}$ be a basis of
		$K$ over $\Q$. 
		There exist constants
		$C_1$ and $C_2$ depending only on 
		the $\gamma_i$'s such that for every $q\in K$
		we can write 
		$q=\displaystyle \sum_{i=1}^d b_i\gamma_i$ where
		$b_i\in \Q$ satisfying $h(b_i)\leq C_1h(q)+C_2$
		for $1\leq i\leq d$.			
		\end{lemma}
		\begin{proof}
		We can solve for the coefficients $b_i$'s using 
		the discriminant of the given basis and the Galois 
		conjugates of $q$ and the $\gamma_i$'s
		to obtain the
		desired height inequality.
		\end{proof}
		
	\subsection{A common setup}\label{subsec:setup}
	The notation and assumptions given in this subsection will
	appear many times. Let $K$,  $k$, 
	$(\alpha_1,\ldots,\alpha_k)$, and 
	$f$ be as in Theorem~\ref{thm:main1}; in particular
	$(\alpha_1,\ldots,\alpha_k)$ is non-degenerate. 
	Assume there is $\theta\in (0,1)$ such that the set $\cM$ 
	defined as in Theorem~\ref{thm:main1} is infinite. Let 
	$\cM_0$ be an infinite subset of $\cM$. By the given inequality on $\displaystyle\max_{1\leq i\leq k} h(q_i)$ and the Northcott property, 
	we observe that there are only finitely many elements of $\cM$ where 
	$n$ is bounded.
	
	The tuple $(\alpha_1,\ldots,\alpha_k)$ is
	said to satisfy Property (P1) if the
	following holds:
	\begin{itemize}
		\item [(P1)] For every $m\in\N$ and $1\leq i\leq k$, $[\Q(\alpha_i):\Q]=[\Q(\alpha_i^m):\Q]$. In other words, 
		if $\beta\neq \alpha_i$ is Galois conjugate
		to $\alpha_i$ then $\displaystyle \frac{\beta}{\alpha_i}\notin \bmmu$.
	\end{itemize}

	The tuple $(\alpha_1,\ldots,\alpha_k)$ is said to satisfy Property
	(P2) if the following holds:
	\begin{itemize}
	\item [(P2)] For $1\leq i, j\leq k$ if
		$\alpha_i \approx \alpha_j$
		then $\alpha_i$ is Galois conjugate to $\alpha_j$.
	\end{itemize}
	Later on, we will use the fact that these properties hold when we replace 
	$(\alpha_1,\ldots,\alpha_k)$ by $(\alpha_1^m,\ldots,\alpha_k^m)$ for some sufficiently large integer $m$. We have:
	\begin{lemma}\label{lem:not in mu}
	If $(\alpha_1,\ldots,\alpha_k)$ satisfies both Properties
	(P1) and (P2) then the following holds. For $1\leq i,j\leq k$, if $\beta$ is Galois conjugate to $\alpha_i$, $\gamma$
	is Galois conjugate to $\alpha_j$, and $\displaystyle\frac{\beta}{\gamma}\in\bmmu$ then $\beta=\gamma$.
	\end{lemma}
	\begin{proof}
	We must have $\alpha_i\approx\alpha_j$, hence Property
	(P2) implies that $\alpha_i$ is Galois conjugate to $\alpha_j$. Hence we can write $\beta=\sigma(\alpha_i)$ and $\gamma=\tau(\alpha_i)$ for some $\sigma,\tau\in G_\Q$. Since
	$\displaystyle\frac{\beta}{\gamma}\in\bmmu$, we have
	$\displaystyle\frac{\tau^{-1}\sigma(\alpha_i)}{\alpha_i}\in\bmmu$.
	Property (P1) gives $\tau^{-1}\sigma(\alpha_i)=\alpha_i$.
	Hence $\beta=\sigma(\alpha_i)=\tau(\alpha_i)=\gamma$.
	\end{proof}

	For the rest of this subsection, assume that $(\alpha_1,\ldots,\alpha_k)$ satisfies both Properties (P1) and (P2). By extending $K$, we may choose a finite subset $S$ of $M_K$ 
	containing $M_K^{\infty}$ such that 
	$\alpha_i$ is an $S$-unit for $1\leq i\leq k$.
	Now we partition $\{\alpha_1,\ldots,\alpha_k\}$
	into $\fs$ subsets
	whose cardinalities are denoted by $m_1,\ldots,m_{\fs}$
	according to the $\fs$ equivalence classes with respect
	to $\approx$. \emph{From now on, we relabel our notation 
	as $\alpha_{i,1},\ldots,\alpha_{i,m_i}$
	where $\{\alpha_{i,1},\ldots,\alpha_{i,m_i}\}$ 
	for $1\leq i\leq \fs$ are the sets in the resulting 
	partition. The $q_1,\ldots,q_k$ are also relabelled as $q_{i,1},\ldots,q_{i,m_i}$ for $1\leq i\leq \fs$ accordingly.} So the original sum $\displaystyle\sum_{i=1}^k q_i\alpha_i^n$
	is now denoted as $\displaystyle\sum_{i=1}^{\fs}\sum_{j=1}^{m_i} q_{i,j}\alpha_{i,j}^n$. We will also denote the
	tuple $(n,q_1,\ldots,q_k)$ by $(n,q_{i,j})_{i,j}$.
	Write $d=[K:\Q]$ and
	let $\{\gamma_1,\ldots,\gamma_d\}$ be a 
	basis of $K$ over $\Q$.

	By Property (P2), for 
	$1\leq i\leq \fs$, the elements
	$\alpha_{i,1},\ldots,\alpha_{i,m_i}$ are
	Galois conjugate to each other and we let 
	$d_i\geq m_i$ denote the number of
	all possible Galois conjugates. We now denote 
	all the other Galois conjugates that do not appear in
	$\{\alpha_{i,1},\ldots,\alpha_{i,m_i}\}$
	as $\alpha_{i,m_i+1},\ldots,\alpha_{i,d_i}$.
	Let $L/\Q$ be the Galois closure of $K/\Q$. For every
	$\sigma\in \Gal(L/\Q)$, for $1\leq i\leq \fs$,
	let $\sigma_{i}$ be the permutation on $\{1,\ldots,d_i\}$
	corresponding to the action of $\sigma$ on $\{\alpha_{i,1},\ldots,\alpha_{i,d_i}\}$. In other words, $\sigma(\alpha_{i,j})=\alpha_{i,\sigma_i(j)}$ for $1\leq j\leq d_i$.

	\subsection{The Key Technical Results}
	Let $K$, $k$, $(\alpha_1,\ldots,\alpha_k)$, $f$, $\theta$, $\cM$, $\cM_0$, $S$, $L$, $d$, 
	and the basis 
	$\{\gamma_1,\ldots,\gamma_d\}$ be as in 
	Subsection~\ref{subsec:setup}; in particular $(\alpha_1,\ldots,\alpha_k)$ satisfies both (P1) and (P2) and each 
	$\alpha_i$ is an $S$-unit. 
	Relabel the
	$\alpha_i$'s and $q_i$'s
	as $\alpha_{i,j}$ and $q_{i,j}$
	for $1\leq i\leq \fs$ and $1\leq j\leq m_i$, and let $d_1,\ldots,d_{\fs}$ and $\alpha_{i,j}$ for $m_i<j\leq d_{i}$ as before. 
	
	\begin{remark}
	Let $x$ be a complex number. When the fractional part of the real part of $x$ is
	$\displaystyle\frac{1}{2}$, there are two integers
	having the closest distance to $x$.
	To be precise, we define
	the closest integer to $x$ to be the \emph{smallest}
	integer $p$ such that $\vert x-p\vert=\Vert x\Vert_{\Z}$. 
	\end{remark}

	\begin{proposition}\label{prop:technical}	
	There exists an infinite subset $\cM_1$ of $\cM_0$
	such that the following holds. For every $(n,q_{i,j})_{i,j}=(n,q_1,\ldots,q_k)\in \cM_1$, 
		let $p$ denote the closest integer
		to $\displaystyle\sum_{i=1}^k q_i\alpha_i^n=\displaystyle\sum_{i=1}^{\fs}\sum_{j=1}^{m_i}q_{i,j}\alpha_{i,j}^n$. We can write
		$p=\displaystyle\sum_{i=1}^{\fs}\sum_{j=1}^{d_i}\eta_{i,j}\alpha_{i,j}^n$ with the following properties:
		\begin{itemize}
			\item [(i)] $\eta_{i,j}\in L$ and $h(\eta_{i,j})=o(n)$ for $1\leq i\leq \fs$
			and $1\leq j\leq d_i$.
			\item [(ii)] For every $\sigma\in \Gal(L/\Q)$ 
			and $1\leq i\leq \fs$, let
			$\sigma_i$ denote the induced permutation
			on $\{1,\ldots,d_i\}$ as in 
			Subsection~\ref{subsec:setup}. We have $\sigma(\eta_{i,j})=\eta_{i,\sigma_i(j)}$ for $1\leq i\leq\fs$
			and $1\leq j\leq d_i$.
			\item [(iii)] $q_{i,j}=\eta_{i,j}$ for $1\leq i\leq \fs$ and $1\leq j\leq m_i$.
			\item [(iv)] Let $B$ be the set of $\beta$ such that $\beta\notin\{q_{i,j}\alpha_{i,j}^n:\ 1\leq i\leq \fs,\ 1\leq j\leq m_i\}$ and $\beta$ is Galois conjugate
			to $q_{i,j}\alpha_{i,j}^n$ for some $1\leq i\leq \fs$ and 
			$1\leq j\leq m_i$. Then the elements
			$\eta_{i,j}\alpha_{i,j}^n$ for $1\leq i\leq \fs$
			and $m_i<j\leq d_i$ are distinct and are exactly
			all the elements of $B$. 
		\end{itemize}
	\end{proposition}
	\begin{proof}
	For every $v\in M_L^{\infty}$, fix $\sigma_v\in \Gal(L/\Q)$
	 such that $v$ corresponds to the (real or complex)
	 embedding $\sigma_v^{-1}$. In other words, for every 
	 $\alpha\in L$, we have:
	 $$\vert \alpha\vert_v=\vert \sigma_v^{-1}(\alpha)\vert^{d(v)/[L:\Q]}.$$
      
	 For $1\leq i\leq \fs$ and for $v\in M_L^{\infty}$,
	 let $\sigma_{v,i}$ be the induced permutation on $\{1,\ldots,d_i\}$ as in Subsection~\ref{subsec:setup} which means 
	 $\sigma_v(\alpha_{i,j})=\alpha_{i,\sigma_{v,i}(j)}$.

	For $(n,q_{i,j})_{i,j}\in \cM_0$, for $1\leq i\leq \fs$
	and $1\leq j\leq m_i$, write:
	$$q_{i,j}=\displaystyle\sum_{\ell=1}^{d}b_{i,j,\ell}\gamma_{\ell}$$
	where $b_{i,j,\ell}\in \Q$. Let $p$ be the closest integer to
	$$\displaystyle\sum_{i=1}^\fs\sum_{j=1}^{m_i} q_{i,j}\alpha_{i,j}^n=\sum_{i=1}^\fs\sum_{j=1}^{m_i}\sum_{\ell=1}^{d}\gamma_{\ell}b_{i,j,\ell}\alpha_{i,j}^n.$$
	Since $\displaystyle\sum_{v\in M_L^\infty} d(v)=[L:\Q]$ 
	and by the definition of $\cM$, we have:
	\begin{equation}
	\begin{split}\label{eq:long 1}
	\theta^n&>\left\vert\sum_{i=1}^\fs\sum_{j=1}^{m_i}\sum_{\ell=1}^{d}\gamma_{\ell}b_{i,j,\ell}\alpha_{i,j}^n-p\right\vert=\displaystyle
	\prod_{v\in M_L^\infty}\left\vert\sum_{i=1}^\fs\sum_{j=1}^{m_i}\sum_{\ell=1}^{d}\gamma_{\ell}b_{i,j,\ell}\alpha_{i,j}^n-p\right\vert^{\frac{d(v)}{[L:\Q]}}\\
	&=\displaystyle\prod_{v\in M_L^{\infty}}\left\vert\sum_{i=1}^\fs\sum_{j=1}^{m_i}\sum_{\ell=1}^{d}
	\sigma_v(\gamma_{\ell})b_{i,j,\ell}\alpha_{i,\sigma_{v,i}(j)}^n-p\right\vert_v. 
	\end{split}
	\end{equation}

	Let $\scrI$ be the set of quadruples $(i,j_1,j_2,\ell)$
	satisfying the following: $1\leq i\leq \fs$, 
	$1\leq j_1\leq m_i$, $1\leq j_2\leq d_i$, 
	and $1\leq \ell\leq d$. For each $(n,q_{i,j})_{i,j}\in \cM_0$, we associate a vector $\bfy:=\bfy(n,q_{i,j})_{i,j}$ whose components are 
	indexed by $\scrI$ and defined to be
	$y_{(i,j_1,j_2,\ell)}=b_{i,j_1,\ell}\alpha_{i,j_2}^n$ 
	for $(i,j_1,j_2,\ell)\in \scrI$. For $v\in M_L^\infty$
	and $\bfa=(i,j_1,j_2,\ell)\in\scrI$, define 
	$$\delta_{v,\bfa}=\twopartdef{\sigma_v(\gamma_\ell)}{\sigma_{v,i}(j_1)=j_2}{0}{\sigma_{v,i}(j_1)\neq j_2}.$$
	Inequality \eqref{eq:long 1} can be
	rewritten as:
	\begin{equation}\label{eq:long 2 v5}
	\prod_{v\in M_L^\infty}\left\vert \sum_{\bfa\in \scrI} \delta_{v,\bfa}y_{\bfa}-p\right\vert_v <\theta^n. 
	\end{equation}

	Among all infinite subsets
	of $\cM_0$, let $\cM_0'$ be such that the vector space
	over $L$ generated by the set
	$\{\bfy(n,q_{i,j})_{i,j}:\ (n,q_{i,j})_{i,j}\in\cM_0'\}$ 
	has minimal dimension. Let $V$ denote this vector space
	and let $\fr=\dim_L(V)$. By applying Gaussian elimination 
	to a system of linear equations 
	defining $V$, we obtain a 
	subset $\scrI^*$ of $\scrI$ consisting of 
	$\fr$ elements together with scalars
	$c_{\bfa,\bfb}\in L$ for $\bfa\in \scrI\setminus \scrI^*$
	and $\bfb\in \scrI^*$
	such that $V$ is given by the linear system:
	$$Y_{\bfa}=\sum_{\bfb\in \scrI^*} c_{\bfa,\bfb}Y_{\bfb}\ \ \text{for $\bfa\in \scrI\setminus\scrI^*$.}$$
	Therefore $y_{\bfa}=\displaystyle \sum_{\bfb\in \scrI^*}c_{\bfa,\bfb}y_{\bfb}$ for $\bfa\in \scrI\setminus \scrI^*$. Consequently, for every $v\in M_L^\infty$, we can write
	$\displaystyle \sum_{\bfa\in \scrI} \delta_{v,\bfa}y_{\bfa}=\sum_{\bfb\in \scrI^*} \tilde{c}_{v,\bfb}y_{\bfb}$
	where the $\tilde{c}_{v,\bfb}$'s for $\bfb\in \scrI^*$
	depend only on $\scrI$, $\scrI^*$, $v$, the $\delta_{v,\bfa}$'s, and the $c_{\bfa,\bfb}$'s. Inequality \eqref{eq:long 2 v5} is rewritten
	as:
	\begin{equation}\label{eq:long 3 v5}
	\prod_{v\in M_L^\infty}\left\vert \sum_{\bfb\in \scrI^*} \tilde{c}_{v,\bfb}y_{\bfb}-p\right\vert_v <\theta^n. 
	\end{equation}

	We are ready to apply the Subspace Theorem as follows. 
	The vector of variables $\bfX$ consists of $X_0$ and $X_{\bfb}$ for $\bfb\in \scrI^*$.
	Let $S'$ be the set of places of $L$ lying above
	$S$. For $w\in S'$ the collection $\cL_{w}$
	of $\vert \scrI^*\vert+1=\fr+1$ linear forms, whose members are denoted  $L_{w,0}(\bfX)$ and $L_{w,\bfb}(\bfX)$
	for $\bfb\in \scrI^*$, is defined as follows:
	\begin{itemize}
		\item If $w\in M_L^{0}$, define $L_{w,0}(\bfX)=X_0$
		and $L_{w,\bfb}(\bfX)=X_{\bfb}$.
	
		\item If $w\in M_L^{\infty}$, define 
		$L_{w,\bfb}(\bfX)=X_{\bfb}$ and 
		define:
		$$L_{w,0}(\bfX)= \displaystyle\sum_{\bfb\in\scrI^*}\tilde{c}_{w,\bfb}X_{\bfb}- X_0.$$
	\end{itemize}
	Clearly, the linear forms in $\cL_w$ are linearly
	independent for every $w\in S'$.
	
	 For $(n,q_{i,j})_{i,j}\in \cM_0'$, we define the vector
	 $\bfx$ whose coordinates are denoted as $x_0$ and 
	 $x_{\bfb}$ for $\bfb\in \scrI^*$ as follows:
	 $$x_0=p\ \text{and}\ x_{\bfb}=y_{\bfb}.$$
	 For $1\leq i\leq \fs$ and $1\leq j\leq m_i$, since
	 $H(q_{i,j})<e^{f(n)}=e^{o(n)}$, we have
	 $\vert q_{i,j}\vert =e^{o(n)}$ and $H(b_{i,j,\ell})=e^{o(n)}$ 
	 for $1\leq \ell\leq d$ by 
	 Lemma~\ref{lem:coefficients}. 
	 By \eqref{eq:long 1}, we have:
	 $$\vert p\vert \leq \left\vert\sum_{i=1}^{\fs}\sum_{j=1}^{m_i} q_{i,j}\alpha_{i,j}^n\right\vert+1\leq C_3^n$$	
	 for some constant $C_3$ depending 		 
     only on $K$, the $\alpha_{i,j}$'s,
     and the sublinear function $f$. Hence there is
     a constant $C_4$ depending only on 
     $K$, $S$, $f$, the $\alpha_{i,j}$'s,
     and the $\gamma_{\ell}$'s such that:
     \begin{equation}\label{eq:C_4}
     H(\tilde{\bfx})<C_4^n. 		 
	\end{equation}	 		 
	 		 
     Let $\scrB=\max\{H(b_{i,j,\ell}):\ 1\leq i\leq \fs, 1\leq j\leq d_i, 1\leq \ell\leq d\}=e^{o(n)}$ and $N=d(d_1+\ldots+d_\fs)$ which
     is the number of triples $(i,j,\ell)$. By inequality
     \eqref{eq:long 3 v5}, the product formula (applied to the $S$-units $\alpha_{i,j}$'s), and
     the integrality of $p$, we have:
     \begin{equation}\label{eq:new Main Subspace}
     \prod_{w\in S'}\prod_{L\in \cL_w} \frac{\vert L(\bfx)\vert_w}{\vert\bfx\vert_w}
     \leq\frac{\theta^n \scrB^{\vert \scrI^*\vert}}{\left(\displaystyle\prod_{w\in S'}\vert\bfx\vert_w\right)^{\vert\scrI^*\vert+1}}=\theta^n\scrB^{\fr}\left(\displaystyle\prod_{w\in S'}\vert\bfx\vert_w\right)^{-\fr-1}.
     \end{equation}
	
	As in the proof of Proposition~\ref{prop:general SML},
	we have:
	$$\displaystyle\prod_{w\in S'} \vert\bfx\vert_w=\frac{H(\tilde{\bfx})}{\displaystyle\prod_{w\notin S'} \vert\bfx\vert_w}\geq \frac{H(\tilde{\bfx})}{\scrB^N}.$$
	Together with \eqref{eq:new Main Subspace}, we have:
	\begin{equation}\label{eq:new Main Subspace1}
     \prod_{w\in S'}\prod_{L\in \cL_w} \frac{\vert L(\bfx)\vert_w}{\vert\bfx\vert_w}
     \leq \theta^n\scrB^{\fr+N(\fr+1)}H(\tilde{\bfx})^{-\fr-1}.
     \end{equation}
	
	From $\theta\in (0,1)$, $\scrB=e^{o(n)}$,
	together with \eqref{eq:C_4} and 
	\eqref{eq:new Main Subspace1},
	there is $\epsilon>0$ such that:
	$$\prod_{w\in S'}\prod_{L\in \cL_w} \frac{\vert L(\bfx)\vert_w}{\vert\bfx\vert_w}<H(\tilde{\bfx})^{-\fr-1-\epsilon}$$
	when $n$ is sufficiently large. 
	By Theorem~\ref{thm:Subspace}, there exists
	a non-trivial linear relation:
	$$a_0X_0+\sum_{\bfb\in\scrI^*} a_{\bfb}X_{\bfb}=0$$
	satisfied by infinitely many $(n,q_{i,j})_{i,j}\in \cM_0'$ where
	$a_0,a_{\bfb}\in L$. In other
	words, there is an infinite subset $\cM_2$ of $\cM_0'$
	such that:
	\begin{equation}\label{eq:relation on M2}
	a_0p+\sum_{\bfb\in\scrI^*}a_{\bfb}y_{\bfb}=0
	\end{equation}
	holds for every $(n,q_{i,j})_{i,j}\in \cM_2$. 
	We claim that $a_0\neq 0$.
	
	Indeed, if we have $a_0=0$, the nontrivial linear relation 
	$\displaystyle\sum_{\bfb\in\scrI^*} a_{\bfb}y_{\bfb}=0$
	implies that the vector space over $L$ generated by
	$\{\bfy(n,q_{i,j})_{i,j}:\ (n,q_{i,j})_{i,j}\in \cM_2\}$
	has dimension at most $\fr-1$. This violates the minimality
	of $\fr$.
	
	Therefore, we can write $p=\displaystyle\sum_{\bfb\in\scrI^*}\frac{-a_\bfb}{a_0}y_{\bfb}$. By the definition of
	$y_{\bfb}$ and the fact that $h(b_{i,j,\ell})=o(n)$
	for every $(i,j,\ell)$,
	we can write: 
	\begin{equation}\label{eq:p=}
	p=\sum_{i=1}^{\fs}\sum_{j=1}^{d_i}\eta_{i,j}\alpha_{i,j}^n
	\end{equation}
	with $\eta_{i,j}\in L$ satisfying $h(\eta_{i,j})=o(n)$
	for every $(i,j)$. Thus we can achieve property (i) in the proposition.
	
	For $1\leq i\leq \fs$ and $m_i<j\leq d_i$, we define $q_{i,j}=0$ formally. Then for every 
	$(n,q_{i,j})_{i,j}\in \cM_2$, we have:
	\begin{equation}\label{eq:qij-etaij}
	\left\vert \sum_{i=1}^{\fs}\sum_{j=1}^{d_i}(q_{i,j}-\eta_{i,j})\alpha_{i,j}^n\right\vert=\left\vert \sum_{i=1}^{\fs}\sum_{j=1}^{m_i}q_{i,j}\alpha_{i,j}^n-p\right\vert<\theta^n
	\end{equation}
	and
	\begin{equation}\label{eq:p with sigma}
		p=\sum_{i=1}^{\fs}\sum_{j=1}^{d_i}\sigma(\eta_{i,j})\alpha_{i,\sigma_i(j)}^n=\sum_{i=1}^{\fs}\sum_{j=1}^{d_i}\sigma(\eta_{i,\sigma_i^{-1}(j)})\alpha_{i,j}^n\ \ \text{for $\sigma\in\Gal(L/\Q)$.}
	\end{equation}
	Here we note that $\sigma_{i}^{-1}$ is the inverse
	of the permutation
	$\sigma_{i}$ on $\{1,\ldots,d_i\}$ which is also
	the permutation induced from $\sigma^{-1}$.
	From \eqref{eq:p=} and \eqref{eq:p with sigma}, we have:
	\begin{equation}\label{eq:difference=0}
	\sum_{i=1}^{\fs}\sum_{j=1}^{d_i}(\eta_{i,j}-\sigma(\eta_{i,\sigma_i^{-1}(j)}))\alpha_{i,j}^n=0\ \ \text{for $\sigma\in\Gal(L/\Q)$.}
	\end{equation}
	
	For $\sigma\in\Gal(L/\Q)$, $1\leq i\leq \fs$, $1\leq j\leq d_i$, let $\cM_2(\sigma,i,j)$ denote the set of $(n,q_{i,j})_{i,j}\in \cM_2$ such that $\eta_{i,j}-\sigma(\eta_{i,\sigma_i^{-1}(j)})\neq 0$. By Lemma~\ref{lem:not in mu} and Proposition~\ref{prop:general SML}, the set $\cM_2(\sigma,i,j)$ is finite. Hence there is a subset $\cM_2'$ of $\cM_2$ such that $\cM_2\setminus\cM_2'$ is finite and $\eta_{i,j}=\sigma(\eta_{i,\sigma_i^{-1}(j)})$ for every $i$, $j$, and $\sigma$ for every
	$(n,q_{i,j})_{i,j}\in\cM_2'$. This gives property (ii).
	
	For $1\leq i\leq \fs$, $1\leq j\leq m_i$, let
	$\cM_2'(i,j)$ be the set of
	$(n,q_{i,j})_{i,j}\in \cM_2'$ such that 
	$q_{i,j}-\eta_{i,j}\neq 0$. 
	We claim that $\cM_2'(i,j)$
	is finite. Assume otherwise that some $\cM_2'(i^*,j^*)$ is 
	infinite. Then there is a \emph{fixed} set $\scrP$
	such that: 
	$$\scrP=\{(i,j):\ 1\leq i\leq \fs, 
	1\leq j\leq d_i, q_{i,j}-\eta_{i,j}\neq 0\}$$
	for \emph{infinitely many} $(n,q_{i,j})_{i,j}\in \cM_2'(i^*,j^*)$. In
	particular $(i^*,j^*)\in \scrP$. Inequality 
	\eqref{eq:qij-etaij} gives:
	\begin{equation}\label{eq:over scrP}
	\left\vert \sum_{(i,j)\in\scrP}(q_{i,j}-\eta_{i,j})\alpha_{i,j}^n\right\vert<\theta^n
	\end{equation}	
	for infinitely many $(n,q_{i,j})_{i,j}\in \cM_2'(i^*,j^*)$. 
	We have
	$\vert(q_{i^*,j^*}-\eta_{i^*,j^*})\alpha_{i,j}^n\vert\geq \vert q_{i^*,j^*}-\eta_{i^*,j^*}\vert$
	since $\vert\alpha_{i^*,j^*}\vert\geq 1$. For
	 $(i,j)\in \scrP$ and
	 $\delta\in (0,1)$, from 
	 $H(1/(q_{i,j}-\eta_{i,j}))=
	 H(q_{i,j}-\eta_{i,j})=e^{o(n)}$ , we have that
	 $\vert q_{i,j}-\eta_{i,j}\vert>\delta^n$
	 when $n$ is sufficiently large.
	 From \eqref{eq:over scrP}, we can apply 
	Proposition~\ref{prop:CZLemma1} to have a non-trivial
	linear relation among the $(q_{i,j}-\eta_{i,j})\alpha_{i,j}^n$ for $(i,j)\in\scrP$ for infinitely many $(n,q_{i,j})_{i,j}\in \cM_2'(i^*,j^*)$. Proposition~\ref{prop:general SML} and
	Lemma~\ref{lem:not in mu} now yield a contradiction. Hence
	$\cM_2'(i,j)$ is finite for every $1\leq i\leq \fs$ and 
	$1\leq j\leq m_i$. Therefore there is a subset $\cM_2''$ of $\cM_2'$
	such that $\cM_2'\setminus\cM_2''$ is finite and $q_{i,j}=\eta_{i,j}$ for $1\leq i\leq \fs$ and $1\leq j\leq m_i$
	for every $(n,q_{i,j})_{i,j}\in\cM_2''$. This gives
	property (iii).
	
	From the choice of $\cM_2'$ and $\cM_2''$, we have that
	$\eta_{i,j}\neq 0$ for every $1\leq i\leq \fs$
	and $1\leq j\leq d_i$. Now if $\eta_{i_1,j_1}\alpha_{i_1,j_1}^n=\eta_{i_2,j_2}\alpha_{i_2,j_2}^n$
	for some $(i_1,j_1)\neq (i_2,j_2)$
	then $H(\alpha_{i_1,j_1}/\alpha_{i_2,j_2})^n=H(\eta_{i_1,j_1}/\eta_{i_2,j_2})=e^{o(n)}$
	which gives a contradiction as $n$ is sufficiently large
	thanks to Lemma~\ref{lem:not in mu}. Hence there is 
	a subset $\cM_1$ of $\cM_2''$ such that 
	$\cM_2''\setminus\cM_1$ is finite and 
	the $\eta_{i,j}\alpha_{i,j}^n$'s are distinct for every
	$(n,q_{i,j})_{i,j}\in\cM_1$.
	
	Given $(n,q_{i,j})\in \cM_1$, let $B$ be the set defined in Property (iv) of the proposition. For $1\leq i\leq \fs$ and 
	$m_i<j\leq d_i$, there is $\sigma\in \Gal(L/\Q)$ such that 
	$\sigma(\alpha_{i,1})=\alpha_{i,j}$; consequently
	$\sigma_i(1)=j$. We have:
	$$\sigma(q_{i,1}\alpha_{i,1}^n)=\sigma(\eta_{i,1})\alpha_{i,j}^n=\eta_{i,j}\alpha_{i,j}^n.$$
	Hence $\eta_{i,j}\alpha_{i,j}^n\in B$.
	
	Conversely if $\beta\in B$, write $\beta=\sigma(q_{i,j}\alpha_{i,j}^n)$ for some $1\leq i\leq \fs$, $1\leq j\leq m_i$
	and $\sigma\in\Gal(L/\Q)$. Then $\beta=\eta_{i,\sigma_i(j)}\alpha_{i,\sigma_i(j)}^n$. By the definition of $B$, we must have $m_i<\sigma_i(j)\leq d_i$.
	
	Therefore the set $\cM_1$ satisfies all the required properties (i)-(iv). This finishes the proof of the proposition.
	\end{proof}
	
	\section{Proof of Theorem~\ref{thm:main1}}\label{sec:proof main1}
	Let $K$, $(\alpha_1,\ldots,\alpha_k)$, $f$, $\theta$, and $\cM$ be as in Theorem~\ref{thm:main1}.
	By extending $K$, we may assume that 
	there is a finite subset $S$ of $M_K$ containing
	$M_K^{\infty}$ such that
	$\alpha_i$ is an $S$-unit 
	for $1\leq i\leq k$. 
	For every $m\in \N$, for $0\leq r\leq m-1$,
	replacing $(\alpha_1,\ldots,\alpha_k)$
	by $(\alpha_1^m,\ldots,\alpha_k^m)$ and replacing
	$\{(n,q_1,\ldots,q_k)\in\cM:\ n\equiv r \bmod m\}$
	by $\left\{\left(\displaystyle\frac{n-r}{m},q_1\alpha_1^r,\ldots,q_k\alpha_k^r\right):\ (n,q_1,\ldots,q_r)\in \cM:\ n\equiv r \bmod m\right\}$ if necessary,
	we may assume that $(\alpha_1,\ldots,\alpha_k)$ satisfies
	Properties (P1) and (P2) defined in 
	Subsection~\ref{subsec:setup}. Let the notation
	$\fs$, $m_1,\ldots,m_\fs$, $\alpha_{i,j}$, and $
	q_{i,j}$ for $1\leq i\leq \fs$ and $1\leq j\leq m_i$
	be as in Subsection~\ref{subsec:setup}. 
	Then let $d_1,\ldots,d_{\fs}$ and $\alpha_{i,j}$
	for $1\leq i\leq \fs$ and $m_i<j\leq d_i$
	be as in Subsection~\ref{subsec:setup}. Let $L$ be
	the Galois closure of $K$ and let $S'$ be the
	set of places of $L$ lying above $S$.
	\subsection{Proof of Property (i) of Theorem~\ref{thm:main1}}
    Assume there is an infinite subset $\cM_0$ of $\cM$
	such that $(q_1\alpha_1^n,\ldots,q_k\alpha_k^n)$ is not
pseudo-Pisot for every $(n,q_1,\ldots,q_k)\in\cM_0$.
	Let $\cM_1$ be an infinite subset of $\cM_0$ satisfying
	the conclusion of Proposition~\ref{prop:technical}. 
	Hence for every $(n,q_{i,j})_{i,j}\in \cM_1$, let $p$ denote
	the closest integer to $\displaystyle\sum_{i=1}^{\fs}\sum_{j=1}^{m_i}q_{i,j}\alpha_{i,j}^n$ then we can write
    $p=\displaystyle\sum_{i=1}^{\fs}\sum_{j=1}^{d_i}\eta_{i,j}\alpha_{i,j}^n$
	where the $\eta_{i,j}$'s satisfy the properties given
	in Proposition~\ref{prop:technical}. Hence:
	\begin{equation}
	\left\vert\sum_{i=1}^{\fs}\sum_{j=m_i+1}^{d_i} \eta_{i,j}\alpha_{i,j}^n\right\vert= \left\vert \displaystyle\sum_{i=1}^{\fs}\sum_{j=1}^{m_i}q_{i,j}\alpha_{i,j}^n-p\right\vert<\theta^n.
	\end{equation}
	By Lemma~\ref{lem:not in mu}, Proposition~\ref{prop:CZLemma1}, Proposition~\ref{prop:general SML},
	and the fact that $H(\eta_{i,j})=e^{o(n)}$, we
	have: 
	$$\max\{\vert\eta_{i,j}\alpha_{i,j}^n\vert:\ 1\leq i\leq\fs, m_i+1\leq j\leq d_i\}<1.$$
	for all but finitely many $(n,q_{i,j})_{i,j}\in \cM_1$.
	By Property (iv) in Proposition~\ref{prop:technical}, we
	have that $(q_{i,j}\alpha_{i,j}^n:\ 1\leq i\leq \fs, 1\leq j\leq m_i)$ is a pseudo-Pisot tuple for all but finitely
	many $(n,q_{i,j})_{i,j}\in\cM_1$. This contradicts the choice of
	$\cM_0$.
	
	\subsection{Proof of Property (ii) of Theorem~\ref{thm:main1}}
	Without loss of generality, assume that $\alpha_{1,1}$
	is not an algebraic integer and we will arrive at
	a contradiction. There is a non-archimedean place $w\in S'$ 
	such that 
	$\vert\alpha_{1,1}\vert_{w}>1$. We now let $\cM_0=\cM$. Let 
	$\cM_1$ be an infinite subset of $\cM_0$ satisfying the
	conclusion of Proposition~\ref{prop:technical}. 
	Hence  
	for every $(n,q_{i,j})_{i,j}\in \cM_1$,  we can write
    $p=\displaystyle\sum_{i=1}^{\fs}\sum_{j=1}^{d_i}\eta_{i,j}\alpha_{i,j}^n$
	where the $\eta_{i,j}$'s satisfy the properties given
	in Proposition~\ref{prop:technical} (hence $\eta_{i,j}\neq 0$ for every $(i,j)$). We have:
	\begin{equation}\label{eq:pw leq 1}
	\left\vert \displaystyle\sum_{i=1}^{\fs}\sum_{j=1}^{d_i}\eta_{i,j}\alpha_{i,j}^n \right\vert_w=\vert p\vert_w\leq 1.
	\end{equation}
	From $H(1/\eta_{1,1})=H(\eta_{1,1})=e^{o(n)}$, we have that
	for every $\delta\in (0,1)$, $\vert\eta_{1,1}\vert_w>\delta^n$
	when $n$ is sufficiently large.
	In particular, for every $1<\Omega<\vert \alpha_{1,1}\vert_w$, we have:
	\begin{equation}\label{eq:max etaijalphaij}
	\max\{\vert \eta_{i,j}\alpha_{i,j}^n\vert_w: 1\leq i\leq \fs, 1\leq j\leq d_i\}\geq \vert\eta_{1,1}\alpha_{1,1}^n\vert_w\geq \Omega^n.
	\end{equation}
	when $n$ is sufficiently large.
	Since $H(\eta_{i,j})=e^{o(n)}$, from inequalities \eqref{eq:pw leq 1}, and \eqref{eq:max etaijalphaij}, we can apply Proposition~\ref{prop:CZLemma1} (with an appropriate $\epsilon$) 
	to conclude that the $\eta_{i,j}\alpha_{i,j}^n$'s satisfy a non-trivial linear relation
	 for infinitely many $(n,q_{i,j})_{i,j}\in \cM_1$. Proposition~\ref{prop:general SML} and Lemma~\ref{lem:not in mu} yield
	 a contradiction.
	 
	 \subsection{Proof of Property (iii) of Theorem~\ref{thm:main1}} 
	 From $H(q_{i,j})=e^{o(n)}$, we have that for sufficiently large $n$, for $1\leq i_1,i_2\leq \fs$, $1\leq j_1\leq m_{i_1}$, and $1\leq j_2\leq m_{j_2}$ if $\sigma(q_{i_1,j_1}\alpha_{i_1,j_1}^n)=q_{i_2,j_2}\alpha_{i_2,j_2}^n$ then $\sigma(\alpha_{i_1,j_1})/\alpha_{i_2,j_2} \in\bmmu$. 
	 For the converse statement, it suffices to prove the following claim: fix $\sigma\in\Gal(L/\Q)$, 
$(i_1,j_1)$, and $(i_2,j_2)$, then for all but finitely many
$(n,q_{i,j})_{i,j}\in\cM$, if $\sigma(\alpha_{i_1,j_1})/\alpha_{i_2,j_2} \in\bmmu$ then $\sigma(q_{i_1,j_1}\alpha_{i_1,j_1}^n)=q_{i_2,j_2}\alpha_{i_2,j_2}^n$. 

	We prove this claim
by contradiction: assume there is an infinite subset
$\cM_0$ violating the conclusion of the claim. By Lemma~\ref{lem:not in mu}, we necessarily have that $\sigma(\alpha_{i_1,j_1})=\alpha_{i_2,j_2}$ and indeed $i_1=i_2:=i^*$ (otherwise, there is no such $\sigma$ and the claim is vacuously true). Hence
$\sigma_{i^*}(j_1)=j_2$. As before,
let $\cM_1$ be an infinite subset of $\cM_0$ as in Proposition~\ref{prop:technical}, then 
$p=\displaystyle\sum_{i=1}^{\fs}\sum_{j=1}^{d_i}\eta_{i,j}\alpha_{i,j}^n$. By Proposition~\ref{prop:technical}, we have:
$$\sigma(q_{i^*,j_1}\alpha_{i^*,j_1}^n)=\sigma(\eta_{i^*,j_1})\alpha_{i^*,j_2}^n=\eta_{i^*,j_2}\alpha_{i^*,j_2}^n=q_{i^*,j_2}\alpha_{i^*,j_2}^n,$$
contradicting the fact that $\cM_0$ does not satisfy the claim.
	
\subsection{Proof of Property (iv) of Theorem~\ref{thm:main1}}\label{subsec:iv}
We need to prove that $\vert\alpha_{i,j}\vert<1$
for $1\leq i\leq \fs$ and $m_i<j\leq d_i$. \emph{In fact, we will
prove the stronger inequality that
$\vert\alpha_{i,j}\vert\leq \theta$
for $1\leq i\leq \fs$ and $m_i<j\leq d_i$.}

Assume $\vert \alpha_{i^*,j^*}\vert >\theta$ for some $1\leq i^*\leq \fs$
and $m_{i^*}< j\leq d_{i^*}$. Let $\cM_0=\cM$ and 
let $\cM_1$ be an infinite subset of 
$\cM_0$
satisfying Proposition~\ref{prop:technical}. As before,
for every $(n,q_{i,j})_{i,j}\in\cM_1$, we have:
	\begin{equation}\label{eq:iv}
	\left\vert\sum_{i=1}^{\fs}\sum_{j=m_i+1}^{d_i} \eta_{i,j}\alpha_{i,j}^n\right\vert= \left\vert \displaystyle\sum_{i=1}^{\fs}\sum_{j=1}^{m_i}q_{i,j}\alpha_{i,j}^n-p\right\vert<\theta^n.
	\end{equation}
	
	Write $N=\displaystyle\sum_{i=1}^{\fs} (d_i-m_i)$. Let $\bfx$ be the vector 
	whose last coordinate is 1 and other coordinates
	are the $\alpha_{i,j}^n$. Then there is a constant $C_5$
	depending only on the $\alpha_{i,j}$'s such that
	$H(\tilde{\bfx})\leq C_5^n$. Choose
	$\epsilon>0$ sufficiently small such that 
	$\theta<\vert\alpha_{i^*,j^*}\vert C_5^{-\epsilon}.$
	Since $H(1/\eta_{i^*,j^*})=e^{o(n)}$, for every $\delta\in (0,1)$,
	we have $\vert \eta_{i^*,j^*}\vert > \delta^n$ when
	$n$ is sufficiently large. Together with
	the property $H(\eta_{i,j})=e^{o(n)}$, we have:
	$$\theta^n<\max\{\vert\eta_{i,j}\alpha_{i,j}^n\vert: 1\leq i\leq \fs, m_i<j\leq d_i\} \left(\prod_{i,j} H(\eta_{i,j})\right)^{-N-1-\epsilon}H(\tilde{\bfx})^{-\epsilon}.$$	
	Together with inequality \eqref{eq:iv}, Proposition~\ref{prop:CZLemma1}, Proposition~\ref{prop:general SML}, and Lemma~\ref{lem:not in mu},
	we obtain a contradiction.

\section{Proof of Corollary~\ref{cor:1} and Corollary~\ref{cor:2}}
	\subsection{Proof of Corollary~\ref{cor:1}}
	Parts (i) and (ii) follow immediately from Theorem~\ref{thm:main1}. In fact, for part (ii), we actually have
	that $\vert\sigma(\alpha_i)\vert\leq \theta$
	thanks to the arguments in Subsection~\ref{subsec:iv}.
	
	For part (iii), by using Theorem~\ref{thm:main1},
	we have that 
	$\sigma(R_i(n)\alpha_i^n)=R_j(n)\alpha_j^n$ for 
	all but finitely many $n\in \cM_\theta$. This immediately
	implies part (iii).
	
	\subsection{Proof of Corollary~\ref{cor:2}}
	As explained before, we may assume that the linear
	recurrence sequence $a_n=Q_1(n)\alpha_1^n+\ldots+Q_k(n)\alpha_k^n$ is non-degenerate and $\vert \alpha_i\vert\geq 1$
	for $1\leq i\leq k$. For every $m\in \N$ and $0\leq r\leq m-1$, by restricting to $n\equiv r$ modulo $m$, we may
	replace $(\alpha_1,\ldots,\alpha_k)$ by 
	$(\alpha_1^m,\ldots,\alpha_k^m)$. Hence we may
	assume that $(\alpha_1,\ldots,\alpha_k)$
	satisfies Properties (P1) and (P2) in Subsection~\ref{subsec:setup}. Let $\fs$ and $m_1,\ldots,m_{\fs}$
	be as before. We also relabel the $\alpha_i$'s and 
	$Q_i$'s as $\alpha_{i,j}$ and $Q_{i,j}$
	for $1\leq i\leq \fs$ and $1\leq j\leq m_i$
	as in Subsection~\ref{subsec:setup}. Corollary~\ref{cor:1}
	gives that for every $\sigma\in G_{\Q}$, $1\leq i\leq \fs$,
	and $1\leq j_1,j_2\leq m_i$, if
	$\sigma(\alpha_{i,j_1})=\alpha_{i,j_2}$ then
	$\sigma(Q_{i,j_1})=Q_{i,j_2}$. In particular, we have
	$Q_{i,j}\in \Q(\alpha_{i,j})[x]$ for every $1\leq i\leq \fs$
	and $1\leq j\leq m_i$. Let $d_1,\ldots,d_\fs$
	and $\alpha_{i,j}$ for $m_i<j\leq d_i$
	be as in Subsection~\ref{subsec:setup}.
	We can now define the polynomials $Q_{i,j}$ for $1\leq i\leq \fs$ and $m_i<j\leq d_i$ as follows: pick any
	$\sigma\in G_{\Q}$ such that $\sigma(\alpha_{i,1})=\alpha_{i,j}$, then define $Q_{i,j}=\sigma(Q_{i,1})$. Therefore
	the linear recurrence sequence:
	$$b_n:=\sum_{i=1}^{\fs}\sum_{j=1}^{d_i}Q_{i,j}(n)\alpha_{i,j}^n=a_n+\sum_{i=1}^{\fs}\sum_{j=m_i+1}^{d_i}Q_{i,j}(n)\alpha_{i,j}^n$$
	is invariant under $G_{\Q}$. Let $D\in \N$ such that
	the coefficients of $DQ_{i,j}$ are algebraic integers
	for all $i,j$. Since the $\alpha_{i,j}$'s are algebraic
	integers by Corollary~\ref{cor:1}, we have $Db_n\in\Z$ for every $n\in\N$. 
	
	Let $\theta_0=\max\{\vert\alpha_{i,j}\vert: 1\leq i\leq \fs, m_i< j\leq d_i\}$. We have that $\theta_0<1$ by Corollary~\ref{cor:1} (in fact, we even have $\theta_0\leq \theta$ by the proof of Corollary~\ref{cor:1}). Let $\tilde{\theta}\in (\theta_0,1)$. When $n$ is sufficiently large, the inequality
	$\Vert a_n\Vert_\Z < \tilde{\theta}^n$
	is equivalent to the condition that $b_n\in\Z$ which is, in turn, equivalent to $Db_n\equiv 0$ modulo $D$. The set of
	such $n$'s is the union of a finite set and finitely many
	arithmetic progressions.

	\bibliographystyle{amsalpha}
	\bibliography{PisotTuple8}

\end{document}